\documentclass{svjour3}                     % onecolumn (standard format)
\smartqed  % flush right qed marks, e.g. at end of propositionof
\usepackage[round]{natbib}
\usepackage{times}
\usepackage{enumerate}
\usepackage{graphicx}
\usepackage{amssymb,amsfonts}
\usepackage{amsmath}
\usepackage{algorithm}
\usepackage{algorithmic}
\usepackage{latexsym}
\usepackage{xcolor}
\usepackage[a4paper,bindingoffset=0.2in,%
            left=1in,right=1in,top=1in,bottom=1in,%
            footskip=.25in]{geometry}
\numberwithin{theorem}{section}
\numberwithin{definition}{section}
\numberwithin{proposition}{section}
\numberwithin{lemma}{section}
\numberwithin{remark}{section}
\numberwithin{example}{section}
\numberwithin{algorithm}{section}

\journalname{Int. J. Appl. Comput. Math. \\ }
\begin{document}

\title{Differentiability and optimality of a fuzzy function}
%\subtitle{Do you have a subtitle?\\ If so, write it here}

\titlerunning{}        % if too long for running head
\author{U. M. Pirzada\thanks{Corresponding author: U. M. Pirzada} \and Debdas Ghosh}

\institute{U. M. Pirzada \at
              School of Engineering and Technology, Navrachana University, Vadodara, India  \\
              \email{salmapirzada@yahoo.com}  \and
           Debdas Ghosh \at
           Department of Mathematical Sciences, Indian Institute of Technology (BHU) Varanasi 221005, India   \\
               \email{debdas.mat@iitbhu.ac.in}
}

\date{}
% The correct dates will be entered by the editor

\maketitle

\begin{abstract}
In this article, we introduce an idea of differentiability for fuzzy functions of fuzzy variables. Explicitly, we define a first order and a second order derivative of a fuzzy function $\tilde{f}: F(\mathbb{R}) \to F(\mathbb{R})$, where $F(\mathbb{R})$ is the set of all fuzzy numbers. In the sequel, we analyze algebra of derivatives of the considered fuzzy functions. With the help of the proposed differentiability notion, we prove a necessary and sufficient condition for optimality to obtain a non-dominated solution of a fuzzy optimization problem. Several numerical examples are given to support the introduced ideas. 
\end{abstract}

%======================================================================================================================
\section{Introduction}
Optimization of fuzzy functions is one of the prominent areas of research in fuzzy mathematics. A considerable number of articles have appeared in this direction. Various types of fuzzy optimization problems are discussed in the classic book that is referred in \cite{LO10}. \cite{LU15} provided a precise and selective look of the existing theory and application of fuzzy optimization. \\

A large number of articles are published on finding optimality conditions for fuzzy optimization problems. Many authors have derived optimality conditions using different notions of differentiability of fuzzy-valued functions. The concept of stationary points for fuzzy optimization problems is studied in \cite{PA08}, \cite{WU107}, \cite{WU209}, and \cite{WU309} under many restrictive situations based on different derivatives. \cite{PI11} studied fuzzy optimization problems concerning a total order relation. A necessary and sufficient optimality condition for unconstrained $L$-fuzzy optimization problems has been proposed in \cite{PI11}. A study on nonlinear unconstrained fuzzy optimization problem is reported in \cite{PA313}. Employing the concept of convexity and Hukuhara differentiability of fuzzy-valued functions, the necessary and sufficient Kuhn-Tucker like optimality conditions for nonlinear fuzzy optimization problems are discussed in \cite{PA211}. Recently, \cite{OS16} studied the necessary and sufficient optimality conditions for fuzzy optimization problems. They derived a necessary optimality condition for a fuzzy function $\tilde{f} : \mathbb{R} \to F(\mathbb{R})$, i.e., for a fuzzy function of real variable. \\

\cite{WU404} proposed an optimal solution concept of the fuzzy optimization problem, which is based on the possibility and necessity measures. Subsequently, duality theories for fuzzy linear programming problems have been studied by \cite{WU503}. \cite{WU603} further introduced a collection of solution concepts of the fuzzy optimization problems using an ordering cone. \cite{WU703} reported a fuzzy-valued Lagrangian function for a fuzzy optimization problem via the idea of a fuzzy scalar product. Further,  \cite{WU703} has also studied the saddle point optimality conditions in the absence of duality gap. With the help of a concept of generalized convexity, sufficient optimality conditions for fuzzy optimization problem have been obtained in \cite{PA110}. The generalized convexity in fuzzy vector optimization through a linear ordering is studied by \cite{AR15}. \\

A Newton method to obtain a non-dominated solution of an unconstrained multi-variable fuzzy optimization problem is proposed by \cite{PI113}. The method is studied with respect to generalized differentiability by \cite{CH15}. A qausi-Newton method is proposed in \cite{GH18} using the concept of generalized differentiability of fuzzy-valued function. Quadratic and cubic interpolation techniques to minimize a univariable fuzzy function have been explored in \cite{GH117} and \cite{GH218}, respectively. \\

From the existing literature on the calculus of fuzzy functions, one can observe that all the derived ideas on the differentiability of fuzzy functions are applicable only for the fuzzy functions of the variables that are real-number-valued. In this paper, we introduce a differentiability concept for a fuzzy function $\tilde{f}: F(\mathbb{R}) \to F(\mathbb{R})$ defined on a fuzzy domain.  The main novelty of this paper is that the variables of the considered fuzzy functions are fuzzy-number-valued. To the best of the authors' knowledge, the concept of differentiability for fuzzy functions with fuzzy-variable is not yet explored in the literature. Using the introduced idea of differentiability, we derive a necessary and sufficient optimality condition for an optimization problem with fuzzy variables.\\

The rest of the paper is organized as follows. In Section \ref{section2}, a few basic definitions related to fuzzy numbers are given. We introduce a new differentiability concept for a fuzzy function using the chain rule in Section \ref{section3}. Besides, we also prove results on the algebra of differentiability in Section \ref{section3}. As an application of the proposed fuzzy differentiation, we prove necessary and sufficient optimality in Section \ref{section4}. Finally, Section \ref{section5} concludes the presented work.

%===========================================================================================================================
\section{Fuzzy numbers and arithmetic}\label{section2}
We start with some basic definitions that are used throughout the paper. We place a tilde bar over the small letters, $\widetilde{a}$, $\widetilde{b}$, $\widetilde{c}$, \dots, to denote fuzzy sets.

\label{sec:2}
\begin{definition}\label{def1}
(\emph{Fuzzy numbers} \cite{PI11}). Let $\mathbb{R}$ be the set of all real numbers and $\tilde{a}:\mathbb{R} \to[0,1]$ be a fuzzy set. We say that $\tilde{a}$ is a fuzzy number if it satisfies the following properties:
\begin{description}
        \item[(i)] {$\tilde{a}$ is normal, i.e., there exists $r_0 \in \mathbb{R} $ such that $\tilde{a}(r_0)=1$, }
        \item[(ii)] {$\tilde{a}$ is fuzzy convex, i.e.,  $\tilde{a}(\beta r +(1- \beta) t)\geq \min \{\tilde{a}(r),\tilde{a}(t)\}$ for any $r, t \in \mathbb{R}$ and $\beta \in [0,1]$,}
        \item[(iii)] {$\tilde{a}(r)$ is upper semi-continuous on $\mathbb{R}$, i.e, $\{ r | \tilde{a}(r) \geq \alpha \}$ is a closed subset of $\mathbb{R}$ for each $\alpha \in (0,1]$, and}
        \item[(iv)] {$closure\{ r \in \mathbb{R} | \tilde{a}(r) >0\}$ is a compact set.}\\
\end{description}
\end{definition}
In particular, for a fuzzy number $\widetilde{a}$ if there exist two real number $l$ and $r$ such that its membership function is given by
\begin{equation*}
\tilde{a}(t) =
\begin{cases}
\tfrac{a-t}{l} & \text{if}~~ a-l \le t \le a \\
\tfrac{t-a}{r} & \text{if}~~ a \le t \le a + r,
\end{cases}
\end{equation*}
then the fuzzy number $\tilde{a}$ is called a \textit{triangular fuzzy number}. We denote this triangular fuzzy number $\widetilde{a}$ by $(a-l,~a,~a+r)$.  \\ \\
The set of all fuzzy numbers on $\mathbb{R}$ is denoted by $F(\mathbb{R})$. \\ \\
For an $\alpha \in (0,1]$, the $\alpha$-\textit{level set} $\tilde{a}_{\alpha}$ of an $\tilde{a}\in F(\mathbb{R})$ is defined by $\tilde{a}_{\alpha} = \{ r \in \mathbb{R}| \tilde{a}(r)\geq \alpha \}$. The 0-level set $\tilde{a}_{0}$ is defined by the closure of the set $\{ r \in \mathbb{R} ~|~ \tilde{a}(r) >0\}$. \\ \\
From the definition of fuzzy numbers, it is readily followed that for any $\tilde{a}\in F(\mathbb{R})$, $\tilde{a}_{\alpha}$ is a compact and convex subset of $\mathbb{R}$ for each $\alpha \in (0,1]$. We thus write $\tilde{a}_{\alpha}=[a_{1}(\alpha),~ a_{2}(\alpha)]$. \\ \\
A fuzzy number $\tilde{a}$ can be recovered from its $\alpha$-level sets by the well-known \textit{decomposition theorem} (see \cite{GE105}), which states that $$\tilde{a}= \bigcup_{\alpha \in [0,1]} \alpha \cdot \tilde{a}_{\alpha},$$ where union on the right-hand side is the standard fuzzy union.

\begin{definition}\label{def2} Let $\tilde{a}, \tilde{b} \in F(\mathbb{R})$ with $\tilde{a}_{\alpha}=[a_{1}(\alpha),a_{2}(\alpha)]$ and  $\tilde{b}_{\alpha}=[b_{1}(\alpha),b_{2}(\alpha)]$. Let $\lambda$ be a real constant. According to Zadeh's extension principle, \textit{addition},  \textit{multiplication} and \textit{scalar multiplication} in the set of fuzzy numbers $F(\mathbb{R})$ are given by their $\alpha$-level sets as follows:
\begin{eqnarray*}
(\tilde{a}\oplus \tilde{b})_{\alpha} & = &
[a_{1}(\alpha) +b_{1}(\alpha), a_{2}(\alpha) + b_{2}(\alpha)] \\
(\tilde{a}\otimes \tilde{b})_{\alpha} & = &
[\min \{ a_{1}(\alpha)b_{1}(\alpha), a_{1}(\alpha)b_{2}(\alpha), a_{2}(\alpha) b_{1}(\alpha), a_{2}(\alpha) b_{2}(\alpha)\} , \\
& & \max \{ a_{1}(\alpha)b_{1}(\alpha), a_{1}(\alpha)b_{2}(\alpha), a_{2}(\alpha) b_{1}(\alpha), a_{2}(\alpha) b_{2}(\alpha)\}] \\
(\lambda \odot \tilde{a})_{\alpha} & = &
[\lambda\cdot a_{1}(\alpha),\lambda\cdot a_{2}(\alpha)],~\text{if}~\lambda \geq 0 \\
			           & = &
[\lambda\cdot a_{2}(\alpha),\lambda\cdot a_{1}(\alpha)],~\text{if}~\lambda < 0,
\end{eqnarray*}
for any $\alpha \in [0,1]$. The $(\tilde{a} \oplus \tilde{b})$,  $(\tilde{a} \otimes \tilde{b})$ and $(\lambda \odot \tilde{a})$ can be determined with the help of the decomposition theorem.
\end{definition}

\begin{definition} (\textit{Difference of fuzzy numbers} \cite{PI113}).
For the pair of fuzzy numbers $\tilde{a}, \tilde{b}$, let $\tilde{a}_{\alpha}=[a_{1}(\alpha),a_{2}(\alpha)]$ and  $\tilde{b}_{\alpha}=[b_{1}(\alpha),b_{2}(\alpha)]$. The difference $\tilde{a} \ominus \tilde{b}$ is defined using its $\alpha$-level sets as
\begin{eqnarray*}
(\tilde{a}\ominus \tilde{b})_{\alpha} = [a_{1}(\alpha) - b_{2}(\alpha),~ a_{2}(\alpha) - b_{1}(\alpha)],
\end{eqnarray*}
for each $\alpha \in [0,1]$. The difference $\tilde{a} \ominus \tilde{b}$ is determined using the decomposition theorem.
\end{definition}

\begin{definition}\label{def4}(\textit{Distance of fuzzy numbers} \cite{WU404}).
Let $A, B \subseteq \mathbb{R}^{n}$. The Hausdorff metric $d_H $ is defined by
\begin{eqnarray*}
d_H(A,B) = \max\left\{\sup_{x \in A}\inf_{y \in B}\|x-y\|,~ \sup_{y \in B}\inf_{x \in A}\|x-y\|\right\}. \end{eqnarray*}
We consider the metric $d_{F}$ on $F(\mathbb{R})$ is given by
\begin{eqnarray*}
d_{F}(\tilde{a}, \tilde{b})  =  \sup_{0 \leq \alpha \leq 1}\{d_H({\tilde{a}}_{\alpha},{\tilde{b}}_{\alpha})\},
\end{eqnarray*}
for all $\tilde{a}, \tilde{b} \in F(\mathbb{R})$. Since $\tilde{a}_{\alpha}$ and $\tilde{b}_{\alpha}$ are compact intervals in $\mathbb{R}$,
\begin{eqnarray*}
d_{F}(\tilde{a}, \tilde{b})  =  \sup_{0 \leq \alpha \leq 1} \max\left\{
|a_{1}(\alpha)-b_{1}(\alpha)|, |a_{2}(\alpha)- b_{2}(\alpha)|\right\}.
\end{eqnarray*}
The set $F(\mathbb{R})$ forms a complete metric space with respect to $d_{F}$.
\end{definition}

\begin{definition}(\textit{LR-fuzzy number}).
An LR-fuzzy number $\tilde{a}$ has membership function of the form

\begin{equation*}
\tilde{a}(t) =
\begin{cases}
L(\tfrac{t- (a^{L}-l)}{l}) & \text{if}~~ a^{L} - l \le t \le a^{L} \\
1 & \text{if}~~ a^{L} \le t \le a^{U} \\
R(\tfrac{(a^{U} + r) - t}{r}) & \text{if}~~ a^{U} \le t \le a^{U} + r,
\end{cases}
\end{equation*}
where $L, R : [0,1] \to [0,1]$ are two non-decreasing shape functions such that $R(0) = L(0) = 0$ and $R(1) = L(1) = 1$. If $L$ and $R$ are invertible functions, then the $\alpha$-level sets are obtained by

\begin{eqnarray*}
\tilde{a}_{\alpha} = [(a^{L} - l) + lL^{-1}(\alpha), (a^{U} + r) - rR^{-1}(\alpha)]
\end{eqnarray*}
The usual LR-fuzzy notation is $\tilde{a} = ((a^{L} -l), a^{L}, a^{U}, (a^{U}+r))_{L,R}$ for fuzzy number. In particular, triangular fuzzy number is represented by $((a - l), a , (a + r))$, where $a^{L}= a^{U} = a$.
\end{definition}

%=====================================================================================================================
\section{Differentiation of the fuzzy functions $\tilde{f}: F(\mathbb{R}) \to F(\mathbb{R})$}\label{section3}
%===================================================================================================================
\subsection{Interpretation of fuzzy functions }
\begin{definition} \label{def5}
A function $\tilde{f}: F(\mathbb{R}) \to F(\mathbb{R})$ is called a \textit{fuzzy function} defined on $F(\mathbb{R})$. For any crisp variable $x$, we have $\tilde{x} \in F(\mathbb{R}) $ and $\tilde{f}(\tilde{x}) \in F(\mathbb{R})$. Thus, for any $\alpha \in [0, 1]$, the $\alpha$-level set of $\tilde{f}(\tilde{x})$ is a closed and bounded interval. \\ \\
For an $\alpha \in [0, 1]$, let the $\alpha$-level set of $\tilde{x}$ be $[x_{1}(x,\alpha), x_{2}(x,\alpha)]$. It means that fuzzy variable $\tilde{x}$ is a fuzzification of crisp variable $x$. The membership function of fuzzy number $\tilde{x}$ is any LR-fuzzy number.  Corresponding to $\tilde{x}$, we present the $\alpha$-level set of $\tilde{f}(\tilde{x})$ by $[f_1(x, \alpha), f_2(x, \alpha)]$. Evidently, $f_1(x, \alpha)$ and $f_2(x, \alpha)$ are two real-valued composite functions: $f_{1} (x,\alpha) = f_{1}(x_{1}, x_{2}, \alpha)$ and $f_{2}(x, \alpha) = f_{2}(x_{1}, x_{2},\alpha)$ on $\mathbb{R}$. We call $f_{1}(x,\alpha)$ and $f_{2} (x,\alpha)$  the $\alpha$-level functions of the fuzzy function $\tilde{f}$. \\
\end{definition}
We consider several examples to illustrate the definition. \\
\begin{example}\label{exam1}
Let $\tilde{f}(\tilde{x}) = \tilde{x} $ be a function defined on $F(\mathbb{R})$, where $\tilde{x} = ((a^{L} -l), a^{L}, a^{U}, (a^{U}+r))_{L,R}$ is a triangular fuzzy number, when $a^{L} = a^{U} = x$, $l = r = 1$ and $L(r) = R(r) = r$, for each $x \in \mathbb{R}$. For an $\alpha \in [0, 1]$, the $\alpha$-level numbers of $\tilde{x}$ are

\[x_{1}(x,\alpha) = (a^{L} - l) + lL^{-1}(\alpha) = ( x - 1) + \alpha 
\] and
\[
x_{2}(x,\alpha) =  (a^{U} + r) - rR^{-1}(\alpha) = ( x + 1 ) - \alpha.\]

The $\alpha$-level functions $f_{1}(x_{1}, x_{2}, \alpha)$ and $f_{2}(x_{1}, x_{2},\alpha)$ are
\[
f_{1}(x,\alpha) = f_{1}(x_{1}, x_{2}, \alpha) = x_{1}(x,\alpha) = (x - 1) + \alpha
\]
and
\[
f_{2}(x,\alpha) = f_{2}(x_{1}, x_{2}, \alpha) = x_{2}(x,\alpha) = ( x + 1 ) + \alpha.
\]
Let $\tilde{x} = \tilde{1} = (0,1,2)$. Then, $x_{1}(1,\alpha) = \alpha$ and $x_{2}(1,\alpha) = (2 - \alpha)$, for $\alpha \in [0,1]$. The $\alpha$-level functions $f_{1}(x, \alpha) = f_{1}(x_{1}, x_{2}, \alpha)$ and $f_{2}(x,\alpha) = f_{2}(x_{1}, x_{2},\alpha)$ evaluated at $(x_{1}, x_{2})$ are
\[
{f}_{1}(x_{1}, x_{2}, \alpha) = x_{1}(1,\alpha) = \alpha
\]
and
\[
{f}_{2}(x_{1}, x_{2}, \alpha) = x_{2}(1,\alpha) = (2-\alpha),
\]
for $\alpha \in [0,1]$.
\end{example}
%===================================================================================================================
\begin{example}\label{exam2}
We consider the fuzzy function $\tilde{f}(\tilde{x}) = \tilde{2} \otimes \tilde{x} $ defined on $F(\mathbb{R})$, where $\tilde{x} = ((a^{L} -l), a^{L}, a^{U}, (a^{U}+r))_{L,R} $, where $a^L = x - 1/2$, $a^{U} = x + 1/2$, $l = u = 1/2$ , for each $x \in \mathbb{R}$ and $L(r)= R(r) = r$. Therefore, $\tilde{x} = (x-1, x-1/2, x+1/2, x+1)$ is a trapezoidal fuzzy number for each $x \in \mathbb{R}$ and $\tilde{2} = (1,2,3)$ is a triangular fuzzy number. For an $\alpha \in [0, 1]$, the $\alpha$-level numbers of $\tilde{x}$ are
\[x_{1}(x,\alpha) = (x-1) + \alpha / 2
\] and
\[
x_{2}(x,\alpha) = (x+1) - \alpha / 2.\]
The $\alpha$-level functions $f_{1}(x_{1}, x_{2}, \alpha)$ and $f_{2}(x_{1}, x_{2},\alpha)$ are
\[
f_{1}(x,\alpha) = f_{1}(x_{1}, x_{2}, \alpha) = (1+\alpha) x_{1}(x,\alpha) = (1+ \alpha) ((x-1) + \alpha /2)
\]
and
\[
f_{2}(x,\alpha) = f_{2}(x_{1}, x_{2}, \alpha) = (3- \alpha) x_{2}(x,\alpha) = (3 - \alpha) ((x+1) - \alpha /2).
\]
Let $\tilde{x} = \tilde{1}$. Then, $x_{1}(1,\alpha) = \alpha /2 $ and $x_{2}(1,\alpha) = 2 - \alpha /2$, for $\alpha \in [0,1]$. The $\alpha$-level functions $f_{1}(x, \alpha) = f_{1}(x_{1}, x_{2}, \alpha)$ and $f_{2}(x,\alpha) = f_{2}(x_{1}, x_{2},\alpha)$ evaluated at $(x_{1}, x_{2})$ are
\[
{f}_{1}(x_{1}, x_{2}, \alpha) = (1+\alpha) \alpha /2
\]
and
\[
{f}_{2}(x_{1}, x_{2}, \alpha) = (3- \alpha) (2-\alpha/2),
\]
for $\alpha \in [0,1]$.
\end{example}

%==================================================================================================================
\begin{example}\label{exam3}
We consider the fuzzy function $\tilde{f}(\tilde{x}) = \tilde{x} \oplus \tilde{3}$ defined on $F(\mathbb{R})$, where $\tilde{x}$ is a gaussian shape fuzzy number defined by membership function $\mu_{\tilde{x}} (r) = \exp\big({-\frac{(r - x)^2}{2 \sigma^2}}\big)$ for each $x \in \mathbb{R}$, where $x$ represents centre and $\sigma$ represents width of fuzzy number and $\tilde{3} = (2, 3, 4)$ is a triangular fuzzy number. For an $\alpha \in (0, 1]$, the $\alpha$-level numbers of $\tilde{x}$, where $\sigma = 1$, are
\[x_{1}(x,\alpha) = x - \sqrt{-2 \sigma^2 \log \alpha}
\] and
\[
x_{2}(x,\alpha) = x + \sqrt{-2 \sigma^2  \log \alpha}.\]
For $\alpha \in (0,1]$, the $\alpha$-level functions $f_{1}(x_{1}, x_{2}, \alpha)$ and $f_{2}(x_{1}, x_{2},\alpha)$ are
\[
f_{1}(x,\alpha) = f_{1}(x_{1}, x_{2}, \alpha) = ( x - \sqrt{-2 \log \alpha}) + ( 2 + \alpha)
\]
and
\[
f_{2}(x,\alpha) = f_{2}(x_{1}, x_{2}, \alpha) = (x + \sqrt{-2 \log \alpha}) + (4 -\alpha).
\]
\end{example}
%===================================================================================================================
\subsection{Continuity of fuzzy functions}

\begin{definition}\label{conti}
A fuzzy function $\tilde{f}: F(\mathbb{R}) \to F(\mathbb{R})$ is said to be continuous at $\tilde{x}_{0} \in F(\mathbb{R})$ if for any $\epsilon > 0$, there exists a $\delta> 0 $ such that $d_{F}(\tilde{f}(\tilde{x}), \tilde{f}(\tilde{x}_{0})) < \epsilon$ whenever $d_{F}(\tilde{x}, \tilde{x}_{0}) < \delta$.
\end{definition}

\begin{theorem}
If $\tilde{f}: F(\mathbb{R}) \to F(\mathbb{R})$ is continuous at $\tilde{x}_{0} \in F(\mathbb{R})$, then its $\alpha$-level functions $f_{1}(x_{1}, x_{2}, \alpha)$ and $f_{2}(x_{1}, x_{2},\alpha)$ are continuous at $(x_{1}(x_{0},\alpha), x_{2}(x_{0}, \alpha))$ for each $\alpha$. Further, if $x_{1}(x,\alpha)$ and $x_{2}(x,\alpha)$ are continuous at $x_{0} \in \mathbb{R}$ for each $\alpha$, then both $f_{1}(x_{1}, x_{2}, \alpha)$ and $f_{2}(x_{1}, x_{2},\alpha)$ are continuous at $x_{0}$.
\end{theorem}
\begin{proof}
Since $\tilde{f}$ is continuous at $\tilde{x}_{0}$, by definition of continuity, for any $\epsilon > 0$ there exists a $\delta > 0$ such that \[
d_{F}(\tilde{f}(\tilde{x}), \tilde{f}(\tilde{x}_{0})) < \epsilon
\] 
whenever $d_{F}(\tilde{x}, \tilde{x}_{0}) < \delta$. Using definition of metric $d_{F}$,

\begin{eqnarray*}
& & \sup_{\alpha} \max\{ | f_{1}(x_{1}(x,\alpha), x_{1}(x,\alpha), \alpha) - f_{1}(x_{1}(x_{0},\alpha), x_{1}(x_{0},\alpha), \alpha )|, \\
& & | f_{2}(x_{1}(x,\alpha), x_{1}(x,\alpha), \alpha) - f_{2}(x_{1}(x_{0},\alpha), x_{1}(x_{0},\alpha), \alpha )| \} < \epsilon
\end{eqnarray*}
This implies 
\[
| f_{1}(x_{1}(x,\alpha), x_{1}(x,\alpha), \alpha) - f_{1}(x_{1}(x_{0},\alpha), x_{1}(x_{0},\alpha), \alpha )| < \epsilon
\]
and
\[
| f_{2}(x_{1}(x,\alpha), x_{1}(x,\alpha), \alpha) - f_{1}(x_{1}(x_{0},\alpha), x_{1}(x_{0},\alpha), \alpha )| < \epsilon
\]
for each $\alpha$ whenever $| x_{1}(x,\alpha) - x_{1}(x_{0}, \alpha)| < \delta $ and $| x_{2}(x,\alpha) - x_{2}(x_{0}, \alpha)| < \delta $. Therefore, $f_{1}$ and $f_{2}$ are continuous $(x_{1}(x_{0}, \alpha), x_{2}(x_{0}, \alpha))$ for each $\alpha$. \\

Now since, $x_{1}(x, \alpha)$ and $x_{2}(x, \alpha)$ are continuous at $x_{0}$, for and $\epsilon_{1}, \epsilon_{2}> 0$ there exists a $\delta_{1} > 0$ such that 
\[
| x_{1} (x, \alpha) - x_{1}(x_{0}, \alpha)| < \epsilon_{1}  
\]
and
\[
| x_{1} (x, \alpha) - x_{1}(x_{0}, \alpha)| < \epsilon_{2}  
\]
whenever $|x - x_{0}| < \delta_{1}$ for each $\alpha$. Take $\epsilon = \min\{\epsilon_{1}, \epsilon_{2} \}$, we have
\[
| f_{1}(x_{1}(x,\alpha), x_{1}(x,\alpha), \alpha) - f_{1}(x_{1}(x_{0},\alpha), x_{1}(x_{0},\alpha), \alpha )| < \epsilon
\]
and
\[
| f_{2}(x_{1}(x,\alpha), x_{1}(x,\alpha), \alpha) - f_{1}(x_{1}(x_{0},\alpha), x_{1}(x_{0},\alpha), \alpha )| < \epsilon,
\]
whenever $|x - x_{0}| < \delta_{1}$ for each $\alpha$. Therefore, $f_{1}$ and $f_{2}$ are continuous at $x_{0}$.

\end{proof}
%\begin{example}
%Consider the fuzzy function $\tilde{f}(\tilde{x}) =  \widetilde{x}$. For a given $\epsilon > 0$, we have to find a suitable $\delta > 0$ so that
%\[
%d_{F}(\tilde{x}, \tilde{x_{0}}) < \delta ~\Longrightarrow~ d_{F}(\tilde{f}(\tilde{x}) , \tilde{f}(\tilde{x_{0}})) < \epsilon.
%\]
%Consider,
%\begin{eqnarray*}
%d_{F}(\tilde{f}(\tilde{x}) , \tilde{f}(\tilde{x_{0}}))) = d_{F}(\tilde{x}, \tilde{x_{0}}) < \delta = \epsilon.
%\end{eqnarray*}
%So choosing $\delta = \epsilon$, we get continuity of $\tilde{f}(\tilde{x}) = \tilde{x}$.
%\end{example}

\begin{example}
Consider the fuzzy function $\widetilde{f}(\widetilde{x}) = (1, 2, 4) \odot \widetilde{x}^2 \oplus (0, 1, 5)$, where $\widetilde{x} = (x-1, x, x + 1)$, $x \ge 1$. \\ \\
For this function, we attempt to check the continuity at $\widetilde{x}_0 = \widetilde{3} = (2,3,4)$. \\ \\
The $\alpha$-level functions of the considered function are
\[f_1(x, \alpha) = \left(1-\alpha\right)\left(x - 1 + \alpha\right)^2 + \alpha ~~\text{and}~~f_2(x, \alpha) = (4 - 2\alpha)\left(x + 1 - \alpha\right)^2 + \left( 5 - 4 \alpha\right).\]
Note that $d_F\left(\widetilde{f}(\widetilde{x}), \widetilde{f}(\widetilde{x}_0)\right) < \epsilon$ whenever
\[|f_1(x, \alpha) - f_1(3, \alpha)|<\epsilon \text{ and } |f_2(x, \alpha) - f_2(3, \alpha)|<\epsilon \text{ for all } \alpha \in [0, 1]. \]
We see that
\[|f_1(x, \alpha) - f_1(3, \alpha)| = |(x - 3) \{(1 - \alpha) (x + 1 + 2 \alpha)\}| \]
and
\[|f_2(x, \alpha) - f_2(3, \alpha)| = |(x - 3) \{(4 - 2 \alpha) (x + 5 - 2\alpha)\}|. \] \\
Observe that, for $|x - 3| < 1$, we have $(1 - \alpha) (x + 1 + 2 \alpha) \le x + 1 + 2 \alpha \le x + 3 < 7$. \\ \\
Therefore, $|f_1(x, \alpha) - f_1(3, \alpha)| < \epsilon$ holds whenever $|x - 3| < \delta_1 := \min \left\{1, \tfrac{\epsilon}{7}\right\}$. \\ \\
Again, note that, for $|x - 3| < 1$, we have $\{(4 - 2 \alpha) (x + 5 - 2\alpha)\} \le 4 (x + 5 - 2 \alpha) \le 4 (x + 5) < 36$. \\ \\
Hence, $|f_2(x, \alpha) - f_2(3, \alpha)| < \epsilon$ holds whenever $|x - 3| < \delta_2 := \min \left\{1, \tfrac{\epsilon}{36}\right\}$. \\ \\
Thus, accumulating all, we see that for $\delta = \min \{\delta_1,~ \delta_2\} = \min\left\{1, \tfrac{\epsilon}{36}\right\}$,
\[d_F\left(\widetilde{f}(\widetilde{x}), \widetilde{f}(\widetilde{x}_0)\right) < \epsilon
\text{ holds whenever }
d_F\left( \widetilde{x} ,  \widetilde{x}_0 \right) < \delta.\]
This shows the continuity of $\widetilde{f}$ at $\widetilde{x}_0$.
\end{example}

%===================================================================================================================
\subsection{Differentiability of fuzzy functions}
\begin{definition}
Let $\tilde{f}: F(\mathbb{R}) \to F(\mathbb{R})$ be a fuzzy function. We say that the the function $\tilde{f}$ is \textit{differentiable} if the following three conditions hold.
\begin{enumerate}[(i)]
\item The functions $x_{1}(x,\alpha) $ and $x_{2}(x,\alpha)$, for $\alpha \in [0,1]$, are differentiable with respect to $x$.
\item Both of $f_{1}(x_{1}, x_{2}, \alpha)$ and $f_{2}(x_{1}, x_{2}, \alpha)$, for each $\alpha \in [0,1]$ are differentiable. Then,
        we denote
        \[ ~~~f_{1}^{\prime}(x_{1}, x_{2}, \alpha) = {{\partial f_{1}} \over {\partial x_{1}}} {{d x_{1}} \over {d x}} +  {{\partial f_{1}} \over {\partial x_{2}}} {{d x_{2}} \over {d x}}~ \text{and}\]
        \[ f_{2}^{\prime}(x_{1}, x_{2}, \alpha) = {{\partial f_{2}} \over {\partial x_{1}}} {{d x_{1}} \over {d x}} +  {{\partial f_{2}} \over {\partial x_{2}}} {{d x_{2}} \over {d x}}. \]
\item The union of the intervals
    \[\bigcup_{\alpha \in [0, 1]} \alpha~\left[\min\{f_{1}^{\prime}(x_{1}, x_{2}, \alpha), f_{2}^{\prime}(x_{1}, x_{2}, \alpha)\},~ \max\{f_{1}^{\prime}(x_{1}, x_{2}, \alpha), f_{2}^{\prime}(x_{1}, x_{2}, \alpha)\}\right]\]
      constitutes a fuzzy number. \\
\end{enumerate}
\end{definition}
Once a fuzzy function $\tilde{f}$ is differentiable at $\tilde{x}_0$, we denote the \textit{derivative value} by $\tilde{f}^{\prime}(\tilde{x}_0)$. Evidently, for $\alpha \in [0, 1]$,
\[\left[\widetilde{f}^{\prime}(\tilde{x}_0)\right]_{\alpha} = \left[\min\{f_{1}^{\prime}(x_{1}, x_{2}, \alpha), f_{2}^{\prime}(x_{1}, x_{2}, \alpha)\},~ \max\{f_{1}^{\prime}(x_{1}, x_{2}, \alpha), f_{2}^{\prime}(x_{1}, x_{2}, \alpha)\}\right].\]
\begin{example}\label{example_fig}
Consider the fuzzy function $\tilde{f}(\tilde{x}) = \tilde{x} \otimes \tilde{x}$, $\tilde{x} = (x-1,x, x+1)$ and $0 \leq x \leq 1$. For this function, the lower $\alpha$-level function is
\[
f_{1}(x_{1}\left(x,\alpha), x_{2}(x,\alpha), \alpha\right) = \min\left\{ x_{1}^2, x_{1} x_{2}, x_{2}^2 \right\},
\]
where $x_{1}(x,\alpha) = (1-\alpha)(x-1) + \alpha x$ and $ x_{2}(x,\alpha) = (1-\alpha)(x+1) + \alpha x$. The function $f_{1}$ is not differentiable with respect to $x$ for each $\alpha$. The graphs of $f_{1}$ for different $\alpha$'s are shown in Fig. 1. From the figure, we see that for each fixed $\alpha$, function $f_{1}(x, \alpha)$ has a corner at different $x$ and therefore it is not differentiable there.    Hence, the fuzzy function $\tilde{f}$ is also not differentiable.
\end{example}

\begin{figure}[h]
\centering
\includegraphics[width=5.0in]{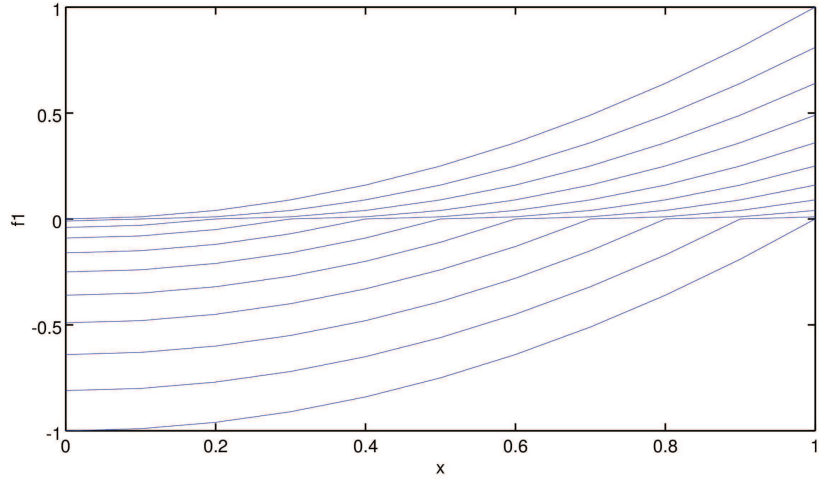}\label{figure1}
\caption{The graph of the function $f_{1}$ in Example \ref{example_fig} for different $\alpha$'s}
\end{figure}

\begin{definition}
A fuzzy function $\tilde{f}: F(\mathbb{R}) \to F(\mathbb{R})$ is said to be \textit{continuously differentiable} if its derivative is a continuous fuzzy function.
\end{definition}
%===========================================================================================================
%We consider some examples.
\begin{example}
Consider the fuzzy function in Example \ref{exam3}. The $\alpha$-level functions $f_{1}(x_{1}, x_{2}, \alpha)$ and $f_{2}(x_{1}, x_{2},\alpha)$ are given by
\[
f_{1}(x_{1}, x_{2}, \alpha) = ( x - \sqrt{-2 \log \alpha}) + ( 2 + \alpha)
\]
and
\[
f_{2}(x_{1}, x_{2}, \alpha) = (x + \sqrt{-2 \log \alpha}) + (4 -\alpha).
\]
Here
\begin{eqnarray*}
f_{1}^{\prime}(x_{1}, x_{2}, \alpha) 
= {{\partial f_{1}} \over {\partial x_{1}}} {{d x_{1}} \over {d x}} +  {{\partial f_{1}} \over {\partial x_{2}}} {{d x_{2}} \over {d x}} 
= {{\partial f_{1}} \over {\partial x_{1}}} 1 + 0 = 1,
\end{eqnarray*}
as ${{\partial f_{1}} \over {\partial x_{1}}} = 1$ and ${{\partial f_{1}} \over {\partial x_{2}}} = 0$, and
\begin{eqnarray*}
f_{2}^{\prime}(x_{1}, x_{2}, \alpha) 
= {{\partial f_{2}} \over {\partial x_{1}}} {{d x_{1}} \over {d x}} +  {{\partial f_{2}} \over {\partial x_{2}}} {{d x_{2}} \over {d x}} 
= 0 + {{\partial f_{2}} \over {\partial x_{2}}}  1 = 1
\end{eqnarray*}
as ${{\partial f_{2}} \over {\partial x_{2}}} = 1$ and ${{\partial f_{2}} \over {\partial x_{1}}} = 0$. \\
The $\alpha$-level sets of the derivative of $\tilde{f}$ are given by
\begin{eqnarray*}
[\tilde{f}^{\prime}(\tilde{x})]_{\alpha} = [f_{1}^{\prime}(x_{1}, x_{2}, \alpha), f_{2}^{\prime}(x_{1}, x_{2}, \alpha)] = [1,1].
\end{eqnarray*}
Hence, derivative of $\tilde{f}(\tilde{x}) = \tilde{x}$ is $1$. \\
\end{example}
%We consider an another fuzzy-valued function.

\begin{example}
Let $\tilde{f}$ be the fuzzy function defined by $\tilde{f}(\tilde{x}) = \tilde{x} \otimes \tilde{x}$, where $\tilde{x} = (x-1, x, x+1)$ is a triangular fuzzy number with $x \geq 1$. The $\alpha$-level numbers for $\tilde{x}$ are
\[x_{1}(x,\alpha) = (1-\alpha)(x-1) + \alpha x
\] and
\[
x_{2}(x,\alpha) = (1-\alpha)(x+1) + \alpha x. \]
The $\alpha$-level functions $f_{1}(x_{1}, x_{2}, \alpha)$ and $f_{2}(x_{1}, x_{2},\alpha)$ are
\[
f_{1}(x_{1}, x_{2}, \alpha) = x_{1}^{2}(x,\alpha) = ((1-\alpha)(x-1) + \alpha x)^2
\]
and
\[
f_{2}(x_{1}, x_{2}, \alpha) = x_{2}^{2}(x,\alpha) = ((1-\alpha)(x+1) + \alpha x)^2.
\]
 Here
\begin{eqnarray*}
f_{1}^{\prime}(x_{1}, x_{2}, \alpha) = {{\partial f_{1}} \over {\partial x_{1}}} {{d x_{1}} \over {d x}} +  {{\partial f_{1}} \over {\partial x_{2}}} {{d x_{2}} \over {d x}} = {{\partial f_{1}} \over {\partial x_{1}}} ((1-\alpha) + \alpha) + 0 = 2 x_{1},
\end{eqnarray*}
as ${{\partial f_{1}} \over {\partial x_{1}}} = 2 x_{1} $ and ${{\partial f_{1}} \over {\partial x_{2}}} = 0$, and
\begin{eqnarray*}
f_{2}^{\prime}(x_{1}, x_{2}, \alpha) = {{\partial f_{2}} \over {\partial x_{1}}} {{d x_{1}} \over {d x}} +  {{\partial f_{2}} \over {\partial x_{2}}} {{d x_{2}} \over {d x}} = 0 + {{\partial f_{2}} \over {\partial x_{2}}} ((1-\alpha) + \alpha) = 2 x_{2}
\end{eqnarray*}
as ${{\partial f_{2}} \over {\partial x_{2}}} = 2 x_{2}$ and ${{\partial f_{2}} \over {\partial x_{1}}} = 0$. \\
The $\alpha$-level sets of the derivative of $\tilde{f}$ are given by
\begin{eqnarray*}
\left[\tilde{f}^{\prime}(\tilde{x})\right]_{\alpha} = [f_{1}^{\prime}(x_{1}, x_{2}, \alpha), f_{2}^{\prime}(x_{1}, x_{2}, \alpha)] =  2 [x_{1}(x,\alpha), x_{2}(x,\alpha)].
\end{eqnarray*}
Hence, the derivative of $\tilde{f}(\tilde{x}) = \tilde{x}$ is $\tilde{f}^{\prime}(\tilde{x}) = 2 \odot \tilde{x}$. \\
\end{example}

\begin{example}
Consider the fuzzy function defined by $\tilde{f}(\tilde{x}) = \exp(\widetilde{-x})$, where $\widetilde{-x}$ is trapezoidal fuzzy number for each $x \in \mathbb{R}$. The $\alpha$-level numbers for $\widetilde{-x}$ are

\[x_{1}(x,\alpha) = (-x - 1) + \alpha /2
\] and
\[
x_{2}(x,\alpha) = (-x + 1) - \alpha /2.\]
The $\alpha$-level functions $f_{1}(x_{1}, x_{2}, \alpha)$ and $f_{2}(x_{1}, x_{2},\alpha)$ are
\[
f_{1}(x_{1}, x_{2}, \alpha) = \exp(x_{1}(x,\alpha))
\]
and
\[
f_{2}(x_{1}, x_{2}, \alpha) = \exp(x_{2}(x,\alpha)).
\]Here
\begin{eqnarray*}
f_{1}^{\prime}(x_{1}, x_{2}, \alpha) 
= {{\partial f_{1}} \over {\partial x_{1}}} {{d x_{1}} \over {d x}} +  {{\partial f_{1}} \over {\partial x_{2}}} {{d x_{2}} \over {d x}} 
= -\exp(x_{1}(x,\alpha)),
\end{eqnarray*}
and
\begin{eqnarray*}
f_{2}^{\prime}(x_{1}, x_{2}, \alpha) 
= {{\partial f_{2}} \over {\partial x_{1}}} {{d x_{1}} \over {d x}} +  {{\partial f_{2}} \over {\partial x_{2}}} {{d x_{2}} \over {d x}} 
= -\exp(x_{2}(x,\alpha)).
\end{eqnarray*}
The $\alpha$-level sets of the derivative of $\tilde{f}$ are given by
\begin{align*}
[\tilde{f}^{\prime}(\tilde{x})]_{\alpha} & = \left[\min\{f_{1}^{\prime}(x_{1}, x_{2}, \alpha), f_{2}^{\prime}(x_{1}, x_{2}, \alpha)\}, \max\{f_{1}^{\prime}(x_{1}, x_{2}, \alpha), f_{2}^{\prime}(x_{1}, x_{2}, \alpha)\}\right] \\
& =  \left[ -\exp(x_{2}(x,\alpha)), -\exp(x_{1}(x,\alpha))\right]
\end{align*}
which, evidently, defines a fuzzy number for each $\tilde{x}$. Hence, derivative of $\tilde{f}(\tilde{x}) = \exp(\widetilde{-x})$ is $\tilde{f}^{\prime}(\tilde{x}) = -\exp(\widetilde{-x})$.
\end{example}
%=============================================================================================

\subsection{Algebra of derivatives of fuzzy functions}
The following theorems describe  some properties of the derivatives of $\tilde{f}:F(\mathbb{R}) \to F(\mathbb{R})$.
\begin{theorem}
Assume that $\tilde{f}, \tilde{g}:F(\mathbb{R}) \to F(\mathbb{R})$ are differentiable functions at $\tilde{x}$. Then $\tilde{f} \oplus \tilde{g}$ and $\tilde{f} \ominus \tilde{g}$ are also differentiable at $\tilde{x}$. The derivatives at $\tilde{x}$ are given by the following formulae:
\begin{enumerate}
\item [(a)] $(\tilde{f} \oplus \tilde{g})^{\prime}(\tilde{x}) = \tilde{f}^{\prime}(\tilde{x}) \oplus \tilde{g}^{\prime}(\tilde{x})$
\item [(b)] $(\tilde{f} \ominus \tilde{g})^{\prime}(\tilde{x}) = \tilde{f}^{\prime}(\tilde{x}) \ominus \tilde{g}^{\prime}(\tilde{x})$.
\end{enumerate}
\end{theorem}
\begin{proof} We prove part (a). Part (b) can be proved similarly. \\ \\
Since the function $\tilde{f}$ is differentiable at $\tilde{x}$, by the definition of derivative of the fuzzy-valued function $\tilde{f}$ we have
\begin{eqnarray*}
[\tilde{f}^{\prime}(\tilde{x})]_{\alpha} = [\min \{f_{1}^{\prime}(x_{1}, x_{2}, \alpha), f_{2}^{\prime}(x_{1}, x_{2}, \alpha)\}, \max\{ f_{1}^{\prime}(x_{1}, x_{2}, \alpha), f_{2}^{\prime}(x_{1}, x_{2}, \alpha)\}], ~\alpha \in [0,1].
\end{eqnarray*}
The derivatives $f_{1}^{\prime}(x_{1}, x_{2}, \alpha)$ and $f_{2}^{\prime}(x_{1}, x_{2}, \alpha)$ are given by
\begin{eqnarray*}
f_{1}^{\prime}(x_{1}, x_{2}, \alpha) = {{\partial f_{1}} \over {\partial x_{1}}} {{d x_{1}} \over {d x}} +  {{\partial f_{1}} \over {\partial x_{2}}} {{d x_{2}} \over {d x}}
\end{eqnarray*}
and
\begin{eqnarray*}
f_{2}^{\prime}(x_{1}, x_{2}, \alpha) = {{\partial f_{2}} \over {\partial x_{1}}} {{d x_{1}} \over {d x}} +  {{\partial f_{2}} \over {\partial x_{2}}} {{d x_{2}} \over {d x}}.
\end{eqnarray*}
The derivative $\tilde{g}^{\prime}(\tilde{x})$ of $\tilde{g}$ is given by
\begin{eqnarray*}
[\tilde{g}^{\prime}(\tilde{x})]_{\alpha} = [\min\{ g_{1}^{\prime}(x_{1}, x_{2}, \alpha), g_{2}^{\prime}(x_{1}, x_{2}, \alpha)\}, \max\{g_{1}^{\prime}(x_{1}, x_{2}, \alpha), g_{2}^{\prime}(x_{1}, x_{2}, \alpha)\}], ~\alpha \in [0,1].
\end{eqnarray*}
The derivatives $g_{1}^{\prime}(x_{1}, x_{2}, \alpha)$ and $g_{2}^{\prime}(x_{1}, x_{2}, \alpha)$ are
\begin{eqnarray*}
g_{1}^{\prime}(x_{1}, x_{2}, \alpha) = {{\partial g_{1}} \over {\partial x_{1}}} {{d x_{1}} \over {d x}} +  {{\partial g_{1}} \over {\partial x_{2}}} {{d x_{2}} \over {d x}}
\end{eqnarray*}
and
\begin{eqnarray*}
g_{2}^{\prime}(x_{1}, x_{2}, \alpha) = {{\partial g_{2}} \over {\partial x_{1}}} {{d x_{1}} \over {d x}} +  {{\partial g_{2}} \over {\partial x_{2}}} {{d x_{2}} \over {d x}}.
\end{eqnarray*}
Now the derivative of $\tilde{f} \oplus \tilde{g}$ at $\tilde{x}$ can be presented as
\begin{equation}\label{eq1}
[(\tilde{f} \oplus \tilde{g})^{\prime}(\tilde{x})]_{\alpha} =  [\min\{(\tilde{f} + \tilde{g})_{1}^{\prime}(x_{1}, x_{2}, \alpha), (\tilde{f} + \tilde{g})_{2}^{\prime}(x_{1}, x_{2}, \alpha)\}, \max\{(\tilde{f} + \tilde{g})_{1}^{\prime}(x_{1}, x_{2}, \alpha), (\tilde{f} + \tilde{g})_{2}^{\prime}(x_{1}, x_{2}, \alpha)\}], ~\alpha \in [0,1]
\end{equation}
provided that the equation defines a fuzzy number and the derivatives  $(\tilde{f} + \tilde{g})_{1}^{\prime}(x_{1}, x_{2}, \alpha)$ and $(\tilde{f} + \tilde{g})_{2}^{\prime}(x_{1}, x_{2}, \alpha)$ exist. \\
We note that
\begin{eqnarray*}
(\tilde{f} + \tilde{g})_{1}^{\prime}(x_{1}, x_{2}, \alpha) & = & {{\partial (\tilde{f} +\tilde{g})_{1}} \over {\partial x_{1}}} {{d x_{1}} \over {d x}} +  {{\partial (\tilde{f} + \tilde{g})_{1}} \over {\partial x_{2}}} {{d x_{2}} \over {d x}} \\
& = & \Big ( {{\partial f_{1}} \over {\partial x_{1}}} + {{\partial g_{1}} \over {\partial x_{1}}} \Big )  {{d x_{1}} \over {d x}} + \Big ( {{\partial f_{1}} \over {\partial x_{2}}} + {{\partial g_{1}} \over {\partial x_{2}}} \Big )  {{d x_{2}} \over {d x}} \\
& = & \Big ( {{\partial f_{1}} \over {\partial x_{1}}} {{d x_{1}} \over {d x}} + {{\partial f_{1}} \over {\partial x_{2}}} {{d x_{2}} \over {d x}} \Big ) + \Big ( {{\partial g_{1}} \over {\partial x_{1}}} {{d x_{1}} \over {d x}} + {{\partial g_{1}} \over {\partial x_{2}}} {{d x_{2}} \over {d x}} \Big ) \\
& = & f_{1}^{\prime}(x_{1}, x_{2}, \alpha) + g_{1}^{\prime}(x_{1}, x_{2}, \alpha)
\end{eqnarray*}
Thus,
\begin{equation}\label{eq11}
(\tilde{f} + \tilde{g})_{1}^{\prime}(x_{1}, x_{2}, \alpha)  =  f_{1}^{\prime}(x_{1}, x_{2}, \alpha) + g_{1}^{\prime}(x_{1}, x_{2}, \alpha).
\end{equation}
Similarly, we have
\begin{equation}\label{eq12}
(\tilde{f} + \tilde{g})_{2}^{\prime}(x_{1}, x_{2}, \alpha) = f_{2}^{\prime}(x_{1}, x_{2}, \alpha) + g_{2}^{\prime}(x_{1}, x_{2}, \alpha).
\end{equation}
Substituting (\ref{eq11}) and (\ref{eq12}), in equation (\ref{eq1}), we have

\begin{eqnarray*}
[(\tilde{f} \oplus \tilde{g})^{\prime}(\tilde{x})]_{\alpha} 
& = & [\min\{(\tilde{f} + \tilde{g})_{1}^{\prime}(x_{1}, x_{2}, \alpha), (\tilde{f} + \tilde{g})_{2}^{\prime}(x_{1}, x_{2}, \alpha)\}, \max\{(\tilde{f} + \tilde{g})_{1}^{\prime}(x_{1}, x_{2}, \alpha), (\tilde{f} + \tilde{g})_{2}^{\prime}(x_{1}, x_{2}, \alpha)\}], \\
& = & [\min\{ f_{1}^{\prime}(x_{1}, x_{2}, \alpha) + g_{1}^{\prime}(x_{1}, x_{2}, \alpha), f_{2}^{\prime}(x_{1}, x_{2}, \alpha) + g_{2}^{\prime}(x_{1}, x_{2}, \alpha)\}, \\
&  & \max\{ f_{1}^{\prime}(x_{1}, x_{2}, \alpha) + g_{1}^{\prime}(x_{1}, x_{2}, \alpha), f_{2}^{\prime}(x_{1}, x_{2}, \alpha) + g_{2}^{\prime}(x_{1}, x_{2}, \alpha) \}],\\
& = & [ \min\{f_{1}^{\prime}(x_{1}, x_{2}, \alpha), f_{2}^{\prime}(x_{1}, x_{2}, \alpha)\}, \max\{ f_{1}^{\prime}(x_{1}, x_{2}, \alpha), f_{2}^{\prime}(x_{1}, x_{2}, \alpha) \}] + \\
&  & [ \min\{g_{1}^{\prime}(x_{1}, x_{2}, \alpha), g_{2}^{\prime}(x_{1}, x_{2}, \alpha)\}, \max\{ g_{1}^{\prime}(x_{1}, x_{2}, \alpha), g_{2}^{\prime}(x_{1}, x_{2}, \alpha) \}]
\end{eqnarray*}
Therefore, we have $(\tilde{f} \oplus \tilde{g})^{\prime}(\tilde{x}) = \tilde{f}^{\prime}(\tilde{x}) \oplus \tilde{g}^{\prime}(\tilde{x})$. \qed
\end{proof}

\begin{theorem}
Assume that $\tilde{f}: F(\mathbb{R}) \to F(\mathbb{R})$ is differentiable at $\tilde{x}$. Then, for any $k \in \mathbb{R}$, $k \odot \tilde{f}$ is also differentiable at $\tilde{x}$ and $(k \odot \tilde{f})^\prime(\tilde{x}) = k \odot \tilde{f}^{\prime}(\tilde{x})$.
\end{theorem}
\begin{proof}
%Since the functions $\tilde{f}$ is differentiable at $\tilde{x}$. By definition, the derivative $\tilde{f}^{\prime}(\tilde{x})$ of the fuzzy-valued function $\tilde{f}$ is
%\begin{eqnarray*}
%[\tilde{f}^{\prime}(\tilde{x})]_{\alpha} = [f_{1}^{\prime}(x_{1}, x_{2}, \alpha), f_{2}^{\prime}(x_{1}, x_{2}, \alpha)], ~\alpha \in [0,1]
%\end{eqnarray*}
%provided that the equation defines a fuzzy number. The derivatives $f_{1}^{\prime}(x_{1}, x_{2}, \alpha)$ and $f_{2}^{\prime}(x_{1}, x_{2}, \alpha)$ are given as follows:
%\begin{eqnarray*}
%f_{1}^{\prime}(x_{1}, x_{2}, \alpha) = {{\partial f_{1}} \over {\partial x_{1}}} {{d x_{1}} \over {d x}} +  {{\partial f_{1}} \over {\partial x_{2}}} {{d x_{2}} \over {d x}}
%\end{eqnarray*}
%and
%\begin{eqnarray*}
%f_{2}^{\prime}(x_{1}, x_{2}, \alpha) = {{\partial f_{2}} \over {\partial x_{1}}} {{d x_{1}} \over {d x}} +  {{\partial f_{2}} \over {\partial x_{2}}} {{d x_{2}} \over {d x}}.
%\end{eqnarray*}
Note that for $k>0$, the derivative of $k \odot \tilde{f} $ exists at $\tilde{x}$ if

\begin{equation}\label{eq2}
[(k \odot \tilde{f})^{\prime}(\tilde{x})]_{\alpha} = \left[\min\{(k\tilde{f})_{1}^{\prime}(x_{1}, x_{2}, \alpha), (k\tilde{f})_{2}^{\prime}(x_{1}, x_{2}, \alpha)\}, \max\{(k\tilde{f})_{1}^{\prime}(x_{1}, x_{2}, \alpha), (k\tilde{f})_{2}^{\prime}(x_{1}, x_{2}, \alpha)\} \right], ~\alpha \in [0,1]
\end{equation}
defines a fuzzy number. \\
We see that
\begin{eqnarray*}
(k\tilde{f})_{1}^{\prime}(x_{1}, x_{2}, \alpha) & = & {{\partial (k\tilde{f})_{1}} \over {\partial x_{1}}} {{d x_{1}} \over {d x}} +  {{\partial (k\tilde{f})_{1}} \over {\partial x_{2}}} {{d x_{2}} \over {d x}} \\
& = & k \Big ( {{\partial f_{1}} \over {\partial x_{1}}} {{d x_{1}} \over {d x}} + {{\partial f_{1}} \over {\partial x_{2}}}   {{d x_{2}} \over {d x}} \Big ) \\
& = & k \tilde{f}_{1}^{\prime}(x_{1}, x_{2}, \alpha).
\end{eqnarray*}
That is,
\begin{equation}\label{eq21}
(k\tilde{f})_{1}^{\prime}(x_{1}, x_{2}, \alpha)  = k f_{1}^{\prime}(x_{1}, x_{2}, \alpha).
\end{equation}
Similarly,
\begin{equation}\label{eq22}
(k\tilde{f})_{2}^{\prime}(x_{1}, x_{2}, \alpha) = k f_{2}^{\prime}(x_{1}, x_{2}, \alpha).
\end{equation}
Substituting (\ref{eq21}) and (\ref{eq22}), in equation (\ref{eq2}), we have  

\begin{eqnarray*}
[(k \odot \tilde{f})^{\prime}(\tilde{x})]_{\alpha} 
& = & [\min\{(k\tilde{f})_{1}^{\prime}(x_{1}, x_{2}, \alpha), (k\tilde{f})_{2}^{\prime}(x_{1}, x_{2}, \alpha)\}, \max\{(k\tilde{f})_{1}^{\prime}(x_{1}, x_{2}, \alpha), (k\tilde{f})_{2}^{\prime}(x_{1}, x_{2}, \alpha)\}], \\
& = & [\min\{ k f_{1}^{\prime}(x_{1}, x_{2}, \alpha), k f_{2}^{\prime}(x_{1}, x_{2}, \alpha) \}, \max\{ k f_{1}^{\prime}(x_{1}, x_{2}, \alpha), k f_{2}^{\prime}(x_{1}, x_{2}, \alpha) \}] \\
& = & k [ \min\{f_{1}^{\prime}(x_{1}, x_{2}, \alpha), f_{2}^{\prime}(x_{1}, x_{2}, \alpha)\}, \max\{f_{1}^{\prime}(x_{1}, x_{2}, \alpha), f_{2}^{\prime}(x_{1}, x_{2}, \alpha) \} ]
\end{eqnarray*}
Hence, we have $(k \odot \tilde{f})^{\prime}(\tilde{x}) = k \odot \tilde{f}^{\prime}(\tilde{x})$. \qed
\end{proof}

% Remark is not needed
%\begin{remark}
%If $k < 0$, the equation (\ref{eq2}) does not define a fuzzy number at $\tilde{x}$, since

%\begin{eqnarray*}
%{{\partial (k\tilde{f})_{1}^{\prime}} \over{\partial \alpha}}(x_{1}, x_{2}, \alpha)  = k {{\partial f_{1}^{\prime}}\over{\partial \alpha}}(x_{1}, x_{2}, \alpha) <0
%\end{eqnarray*}
%and
%\begin{eqnarray*}
%{{\partial (k\tilde{f})_{2}^{\prime}}\over {\partial \alpha}}(x_{1}, x_{2}, \alpha)  = k {{\partial f_{2}^{\prime}}\over {\partial \alpha}}(x_{1}, x_{2}, \alpha) >0.
%\end{eqnarray*}
%\end{remark}

%=====================================================================================================================
\subsection{Second order differentiability}
We define second order differentiability of fuzzy function as follows.
\begin{definition} \label{def3.4}
Let $\tilde{f}: F(\mathbb{R}) \to F(\mathbb{R})$ be a fuzzy function. The second derivative $\tilde{f}^{\prime\prime}(\tilde{x})$ of $\tilde{f}$ is defined by
\begin{eqnarray*}
[\tilde{f}^{\prime\prime}(\tilde{x})]_{\alpha} = [\min\{f_{1}^{\prime\prime}(x_{1}, x_{2}, \alpha), f_{2}^{\prime\prime}(x_{1}, x_{2}, \alpha)\}, \max\{f_{1}^{\prime\prime}(x_{1}, x_{2}, \alpha), f_{2}^{\prime\prime}(x_{1}, x_{2}, \alpha)\}], ~\alpha \in [0,1]
\end{eqnarray*}
provided that the equation defines a fuzzy number. The derivatives $f_{1}^{\prime\prime}(x_{1}, x_{2}, \alpha)$ and $f_{2}^{\prime\prime}(x_{1}, x_{2}, \alpha)$ are defined as follows:
\begin{eqnarray*}
f_{1}^{\prime\prime}(x_{1}, x_{2}, \alpha) & = & {{d} \over {dx}} \Big( {{\partial f_{1}} \over {\partial x_{1}}} \Big) {{d x_{1}} \over {d x}} + {{\partial f_{1}} \over {\partial x_{1}}} {{d^2 x_{1}} \over {dx^2}}  +  {{d} \over {dx}} \Big ( {{\partial f_{1}} \over {\partial x_{2}}}\Big) {{d x_{2}} \over {d x}} + {{\partial f_{1}} \over {\partial x_{2}}} {{d^2 x_{2}} \over {dx^2}} \\
& = & \Big( {{\partial^2 f_{1}} \over {\partial x_{1}^2}} {{d x_{1}} \over {d x}} + {{\partial^2 f_{1}} \over {\partial x_{2}\partial x_{1}}} {{d x_{2}} \over {d x}} \Big) {{d x_{1}} \over {d x}} + {{\partial f_{1}} \over {\partial x_{1}}} {{d^2 x_{1}} \over {dx^2}} \\
& ~~~~~+ & \Big( {{\partial^2 f_{1}} \over {\partial x_{1} \partial x_{2}}} {{d x_{1}} \over {d x}} + {{\partial^2 f_{1}} \over {\partial x_{2}^2}} {{d x_{2}} \over {d x}} \Big) {{d x_{2}} \over {d x}} + {{\partial f_{1}} \over {\partial x_{2}}} {{d^2 x_{2}} \over {dx^2}}
\end{eqnarray*}
and
\begin{eqnarray*}
f_{2}^{\prime\prime}(x_{1}, x_{2}, \alpha) & = & {{d} \over {dx}} \Big( {{\partial f_{2}} \over {\partial x_{1}}} \Big) {{d x_{1}} \over {d x}} + {{\partial f_{2}} \over {\partial x_{1}}} {{d^2 x_{1}} \over {dx^2}}  +  {{d} \over {dx}} \Big ( {{\partial f_{2}} \over {\partial x_{2}}}\Big) {{d x_{2}} \over {d x}} + {{\partial f_{2}} \over {\partial x_{2}}} {{d^2 x_{2}} \over {dx^2}} \\
& = & \Big( {{\partial^2 f_{2}} \over {\partial x_{1}^2}} {{d x_{1}} \over {d x}} + {{\partial^2 f_{2}} \over {\partial x_{2}\partial x_{1}}} {{d x_{2}} \over {d x}} \Big) {{d x_{1}} \over {d x}} + {{\partial f_{2}} \over {\partial x_{1}}} {{d^2 x_{1}} \over {dx^2}} \\
& ~~~~~+ & \Big( {{\partial^2 f_{2}} \over {\partial x_{1} \partial x_{2}}} {{d x_{1}} \over {d x}} + {{\partial^2 f_{2}} \over {\partial x_{2}^2}} {{d x_{2}} \over {d x}} \Big) {{d x_{2}} \over {d x}} + {{\partial f_{2}} \over {\partial x_{2}}} {{d^2 x_{2}} \over {dx^2}}.
\end{eqnarray*}
\end{definition}

\begin{remark}
The second derivative of the fuzzy function $\tilde{f}$ exists if
\begin{itemize}
\item [(i)] both $f_{1}(x_{1}, x_{2}, \alpha)$ and $f_{2}(x_{1}, x_{2}, \alpha)$, for $\alpha \in [0,1]$, have continuous second order partial derivatives, and
\item [(ii)]the functions $x_{1}(x,\alpha) $ and $x_{2}(x,\alpha)$, for $\alpha \in [0,1]$, are twice differentiable with respect to $x$.
\end{itemize}
\end{remark}
\begin{definition}
A fuzzy function $\tilde{f}: F(\mathbb{R}) \to F(\mathbb{R})$ is \textit{twice continuously differentiable} if its first and second derivatives are continuous.
\end{definition}
\begin{example}
Let $\tilde{f}$ be a fuzzy function defined by $\tilde{f}(\tilde{x}) = \exp(\widetilde{-x})$, where $\widetilde{-x}$ is the triangular fuzzy number $ (-x-1, -x, -x+1)$. The $\alpha$-level numbers of $ \widetilde{-x}$ are
\[x_{1}(x,\alpha) = (1-\alpha)(-x-1) - \alpha x = -x - 1 + \alpha
\] and
\[
x_{2}(x,\alpha) = (1-\alpha)(-x+1) - \alpha x = -x + 1 - \alpha,\]
for $\alpha \in [0,1]$ and $\alpha$-level functions $f_{1}(x_{1}, x_{2}, \alpha)$ and $f_{2}(x_{1}, x_{2},\alpha)$ are
\[
f_{1}(x_{1}, x_{2}, \alpha) = \exp(x_{1}(x,\alpha))
\]
and
\[
f_{2}(x_{1}, x_{2}, \alpha) = \exp(x_{2}(x,\alpha)),
\]
for $\alpha \in [0,1]$. Using Definition \ref{def3.4},
\begin{eqnarray*}
f_{1}^{\prime \prime }(x_{1}, x_{2}, \alpha) = \exp(x_{1}(x,\alpha)),
\end{eqnarray*}
and
\begin{eqnarray*}
f_{2}^{\prime \prime }(x_{1}, x_{2}, \alpha) = \exp(x_{2}(x,\alpha)).
\end{eqnarray*}
The $\alpha$-level sets
\begin{eqnarray*}
[\tilde{f}^{\prime\prime}(\tilde{x})]_{\alpha} = [ \exp(x_{1}(x,\alpha)), \exp(x_{2}(x,\alpha))],
\end{eqnarray*}
$\alpha \in [0,1]$ defines a fuzzy number for each $x$. Therefore, the second derivative of $\tilde{f}(\tilde{x}) = \exp(\widetilde{-x})$ is $\tilde{f}^{\prime\prime}(\tilde{x}) = \exp(\widetilde{-x})$.
\end{example}

%=====================================================================================================================

\section{Fuzzy optimization}\label{section4}
A partial order relation on $F(\mathbb{R})$ is defined as follows.
\begin{definition}
For $\tilde{a}, \tilde{b} \in F(\mathbb{R})$, we say that $\tilde{a}$ dominates $\widetilde{b}$, written as $\tilde{a}\preceq \tilde{b}$, if $a_{1}(\alpha) \leq b_{1}(\alpha)$ and $a_{2}(\alpha) \leq b_{2}(\alpha)$, for all $\alpha$. We say that $\tilde{a}$ strictly dominates $\widetilde{b}$, written as $\tilde{a}\prec \tilde{b}$, if $a_{1}(\alpha) \leq b_{1}(\alpha)$ and $a_{2}(\alpha) \leq b_{2}(\alpha)$, for all $\alpha$ and there exists an $\alpha_{0} \in [0,1]$ such that $a_{1}(\alpha_{0}) < b_{1}(\alpha_{0})$ or $a_{2}(\alpha_{0}) < b_{2}(\alpha_{0})$.
\end{definition}
We now define a subset of $F(\mathbb{R})$.
\begin{definition}
A set $X$ is called a subset of $F(\mathbb{R})$ if all the elements of $X$ are also elements of $F(\mathbb{R})$. It is denoted by $X \subset F(\mathbb{R})$.
\end{definition}
\begin{example}
Let $X = \{ \tilde{y}: \tilde{y} = 2\odot \tilde{x}, \tilde{x} \in F(\mathbb{R})\}$. Clearly, each element of $X$ is also an element of $F(\mathbb{R})$.  Therefore, $X \subset F(\mathbb{R})$.
\end{example}
\begin{definition} (\emph{Open ball}).
Let $((F(\mathbb{R}), d_{F})$ be a metric space, $\tilde{x}^{*} \in F(\mathbb{R})$ and $\epsilon > 0$.The fuzzy open ball centered at $\tilde{x}^{*}$ with radius $\epsilon$ is defined to be the set $B(\tilde{x}^{*},\epsilon) = \{ \tilde{x} \in F(\mathbb{R}) : d_{F}(\tilde{x}, \tilde{x}^{*}) < \epsilon \}$.
\end{definition}
\begin{definition}(\textit{Interior point}).
Let $X$ be a subset of the metric space $F(\mathbb{R})$ with metric $d_{F}$. A point $\tilde{x}^{*} \in F(\mathbb{R})$ is called an interior point of $X$ if there exists an $\epsilon > 0$, such that the fuzzy open ball $B(\tilde{x}^{*},\epsilon) \subseteq X$.
\end{definition}
$\\\\$
Let $\tilde{f}: X \subseteq F(\mathbb{R}) \to F(\mathbb{R})$ be a fuzzy function. Consider the \textit{fuzzy optimization problem} (FOP): \\
\[
\text{Minimize}~~\tilde{f}(\tilde{x}),~\tilde{x} \in X.
\]
\begin{definition}
Let $\tilde{x}^{*} \in F(\mathbb{R})$.
\begin{enumerate}
\item The point $\tilde{x}^{*}$ is called a \textit{locally non-dominated solution} of (FOP) if there is no $\tilde{x} \in B(\tilde{x}^{*},~\epsilon) \cap X$ such that $\tilde{f}(\tilde{x}) \prec \tilde{f}(\tilde{x}^{*})$.
\item The point $\tilde{x}^{*} \in F(\mathbb{R})$ is called a \textit{non-dominated solution} of (FOP) if there is no $\tilde{x}\neq \tilde{x}^{*} \in X$ such that $\tilde{f}(\tilde{x}) \prec \tilde{f}(\tilde{x}^{*})$.
\end{enumerate}
\end{definition}

\subsection{Necessary condition for optimality}
%Now we propose a first order necessary optimality condition.
\begin{theorem}\label{thm3}
Let $\tilde{f}: X \subseteq F(\mathbb{R}) \to F(\mathbb{R})$ be differentiable at $\tilde{x}^{*} \in X$, where $\tilde{x}^{*}$ is an interior point of $X$. If $\tilde{x}^{*}$ is a locally non-dominated solution of $\tilde{f}$ then there exists an $\alpha \in [0, 1]$ such that 

\[f_{1}^{\prime}\left(x_{1}(x^{*},\alpha), x_{2}(x^{*},\alpha),\alpha\right) = 0 \] 
or \[ f_{2}^{\prime}\left(x_{2}(x^{*},\alpha), x_{2}(x^{*},\alpha),\alpha\right) = 0 \]
where $x^{*}$ is real number corresponding to the fuzzy number $\tilde{x}^{*}$.
\end{theorem}

\begin{proof}
Suppose the result is not true. That is, for all $\alpha \in [0, 1]$, 

\[f_{1}^{\prime}(x_{1}(x^{*},\alpha), x_{2}(x^{*},\alpha),\alpha) \neq 0 \text{ and } f_{2}^{\prime}(x_{2}(x^{*},\alpha), x_{2}(x^{*},\alpha),\alpha) \neq 0. \]

Since $\tilde{f}$ is differentiable at $\tilde{x}^{*}$, for any $\alpha \in [0, 1]$ we have 

\[f_{1}^{\prime}(x_{1}(x^{*},\alpha), x_{2}(x^{*},\alpha),\alpha) = {{\partial f_{1}} \over {\partial x_{1}}} {{d x_{1}} \over {d x}} +  {{\partial f_{1}} \over {\partial x_{2}}} {{d x_{2}} \over {d x}}. \]

This implies
\[
\nabla f_{1}(x^{*},\alpha) = \left( {{\partial f_{1}} \over {\partial x_{1}}}, {{\partial \tilde{f}_{1}} \over {\partial x_{2}}} \right) \neq (0, 0). \]

Let $\bar{d} = -\nabla f_{1}(x^{*},\alpha)$. Then we get

\[ \nabla  f_{1}(x^{*},\alpha)^T ~ \bar{d} =  - \| \nabla f_{1}(x^{*},\alpha) \|^{2} < 0. \]

By Theorem 4.1.2 in [2], there is a $\delta_1(\alpha) > 0$ such that

\[ f_{1}(x^{*} + \lambda \bar{d},\alpha) < f_{1}(x^{*},\alpha) \text{ for } \lambda \in (0, \delta_1(\alpha)).\]

Let $\xi_1(\alpha) = \sup \left\{\delta_1(\alpha) : f_{1}(x^{*} + \lambda \bar{d},\alpha) < f_{1}(x^{*},\alpha) \text{ for } \lambda \in (0, \delta_1(\alpha))\right\}$ and $\delta_1 = \inf_{\alpha \in [0, 1]} \xi_1(\alpha)$. \\

We note that $\delta_1 = \displaystyle \inf_{\alpha \in [0, 1]} \left\{\xi_1(\alpha) : f_{1}(x^{*} + \lambda \bar{d},\alpha) - f_{1}(x^{*},\alpha) < 0 \text{ for } \lambda \in (0, \xi_1(\alpha))\right\}$. \\

Since $f_{1}(x^{*} + \lambda \bar{d},\alpha)$ is continuous in $\lambda$ and $\alpha$, evidently, $\xi_1(\alpha)$ is a continuous function on $\alpha$. Hence, $\delta_1 >0$. \\

Thus, for any $\alpha \in [0, 1]$ \[ f_{1}(x_{1}(x^{*} + \lambda \bar{d}, \alpha), x_{2}(x^{*} + \lambda \bar{d},\alpha),\alpha) < f_{1}(x_{1}(x^{*},\alpha), x_{2}(x^{*},\alpha),\alpha) \text{ for } \lambda \in (0, \delta_1).\]

In a similar way there exists $\delta_2 > 0$ such that for all $\alpha \in [0, 1]$,  
\[ f_{2}(x_{1}(x^{*} + \lambda \bar{d}, \alpha), x_{2}(x^{*} + \lambda \bar{d},\alpha),\alpha) < f_{2}(x_{1}(x^{*},\alpha), x_{2}(x^{*},\alpha),\alpha) \text{ for } \lambda \in (0, \delta_2).\]

Letting $\delta = \min \{\delta_1, \delta_2\}$, we see that 
\[
\tilde{f}(\tilde{x}^{*} + \lambda \bar{d}) \prec \tilde{f}(x^{*}) \text{ for } \lambda \in (0,\delta). 
\]

This contradicts to our assumption that $x^{*}$ is a locally non-dominated solution of (FOP).  Hence the result follows. \qed
\end{proof}

\begin{definition}
If $\tilde{x}^{*}$ satisfies necessary condition for optimality, then $\tilde{x}^{*}$ is called a stationary point of the fuzzy function $\tilde{f}$.
\end{definition}

\begin{example}\label{ex42}
Consider a fuzzy-valued function $\tilde{f}: F(\mathbb{R}) \to F(\mathbb{R})$ defined by $\tilde{f}(\tilde{x})= \tilde{x} \otimes \tilde{x} \ominus 4 \odot \tilde{x}$, where $\tilde{x} = (x-1, x, x+1)$ and $x \geq 1$. For each $\alpha \in [0,1]$, the $\alpha$-level functions of given fuzzy function are
\begin{eqnarray*}
f_{1}\left(x_{1}(x,\alpha), x_{2}(x, \alpha),\alpha\right)
& = & x_{1}^{2}(x,\alpha) - 4 x_{2}(x, \alpha) \\
& = & ((1-\alpha)(x-1) + \alpha x)^2 - 4 ((1-\alpha)(x+1) + \alpha x),
\end{eqnarray*}
and
\begin{eqnarray*}
f_{2}(x_{1}\left(x,\alpha), x_{2}(x, \alpha),\alpha\right)
& = & x_{2}^{2}(x,\alpha) - 4 x_{1}(x, \alpha) \\
& = & ((1-\alpha)(x+1) + \alpha x)^2 - 4 ((1-\alpha)(x-1) + \alpha x).
\end{eqnarray*}
The derivatives of level functions are
\begin{eqnarray*}
f_{1}^{\prime}(x_{1}, x_{2}, \alpha) = {{\partial f_{1}} \over {\partial x_{1}}} {{d x_{1}} \over {d x}} +  {{\partial \tilde{f}_{1}} \over {\partial x_{2}}} {{d x_{2}} \over {d x}} = {{\partial f_{1}} \over {\partial x_{1}}} ((1-\alpha) + \alpha) + (-4) =   2 x_{1} -4,
\end{eqnarray*}
as ${{\partial f_{1}} \over {\partial x_{1}}} = 2 x_{1} $ and ${{\partial f_{1}} \over {\partial x_{2}}} = -4$, and
\begin{eqnarray*}
f_{2}^{\prime}(x_{1}, x_{2}, \alpha) = {{\partial f_{2}} \over {\partial x_{1}}} {{d x_{1}} \over {d x}} +  {{\partial f_{2}} \over {\partial x_{2}}} {{d x_{2}} \over {d x}} = (-4) + {{\partial f_{2}} \over {\partial x_{2}}} ((1-\alpha) + \alpha)  = 2x_{2} - 4
\end{eqnarray*}
as ${{\partial f_{2}} \over {\partial x_{2}}} = 2 x_{2}$ and ${{\partial f_{2}} \over {\partial x_{1}}} = -4$. \\ \\
Clearly, $\tilde{f}$ is differentiable with $\alpha$-level functions $\tilde{f}_{1}^{\prime}(x_{1}, x_{2}, \alpha) $ and $\tilde{f}_{2}^{\prime}(x_{1}, x_{2}, \alpha)$, and the fuzzy derivative $\tilde{f}^{\prime}(\tilde{x}) = 2 \odot \tilde{x} \ominus 4$. \\ \\
By applying the necessary optimality condition, there exists an $\alpha \in [0,1]$ such that

\[f_{1}^{\prime}(x_{1}(x^{*},\alpha), x_{2}(x^{*}, \alpha), \alpha) = 0 \] or
\[f_{2}^{\prime}(x_{1}(x^{*},\alpha), x_{2}(x^{*}, \alpha), \alpha) = 0. \]

That is, there exists an $\alpha \in [0,1]$ such that

\[ 2x_{1} - 4 = 2 ((1-\alpha)(x-1) + \alpha x) - 4 = 0,\] or
\[ 2x_{2} - 4 = 2 ((1-\alpha)(x+1) + \alpha x) - 4 = 0.\]
Therefore, $x^{*} = 2$. Hence, $\tilde{x}^{*} = \tilde{2} = (1, 2, 3)$ is stationary point of the given fuzzy function.
\end{example}
%=====================================================================================================================

\subsection{Sufficient condition for optimality}
\begin{theorem}
Let $\tilde{f}: X \subset F(\mathbb{R}) \to F(\mathbb{R})$ be a twice continuously differentiable fuzzy function. Suppose that $\widetilde{x}^{*} \in X$ is a stationary point of $\tilde{f}$.
\begin{enumerate}[(i)]
\item \label{part1} If $f_{1}^{\prime\prime}(x_{1},x_{2},\alpha) \geq 0$ for all $\tilde{x} \in X$, then $\tilde{x}^{*}$ is a non-dominated solution of ~$\min_{\tilde{x} \in X}~\tilde{f}$.
\item \label{part2} If $f_{1}^{\prime\prime}(x_{1}(x^{*},1),(x_{2}(x^{*},1),1) > 0$, then $x^{*}$ is a locally non-dominated solution of ~$\min_{\tilde{x} \in X}~\tilde{f}$.
\end{enumerate}
\end{theorem}
\begin{proof}
Since $\tilde{f}$ is twice continuously differentiable, the lower $\alpha$-level function $f_{1}\left(x_{1}(x,\alpha), x_{2}(x,\alpha), \alpha\right)$ is also a twice continuously differentiable function for each $\alpha \in [0, 1]$. \\ \\
Then, by Taylor's expansion of $f_{1}$ at $x^*$, we have for $x \neq x^{*}$,
\begin{eqnarray*}
f_{1}\left(x_{1}(x,\alpha), x_{2}(x,\alpha),\alpha\right) - f_{1}\left(x_{1}(x^{*},\alpha),~ x_{2}(x^{*},~ \alpha),\alpha\right) = \tfrac{1}{2} f_{1}^{\prime\prime}(x_{1}(z, \alpha),x_{2}(z, \alpha), \alpha) (x-x^{*})^2,
\end{eqnarray*}
where $z = x^{*} + \tau(x- x^{*})$, $\tau\in(0, 1)$. \\ \\
Under the hypothesis in Part (\ref{part1}), $f_{1}^{\prime\prime}(x_{1}(z, \alpha),x_{2}(z, \alpha), \alpha) \geq 0$. \\ \\
Thence, we have
\[
f_{1}(x_{1}(x,\alpha), x_{2}(x,\alpha),\alpha) \ge f_{1}(x_{1}(x^{*},\alpha), x_{2}(x^{*},\alpha),\alpha) \text{ for all } x \neq x^*.
\]
Therefore, there cannot be any $\tilde{x}$ in $X$ such that $\tilde{f}(\tilde{x})$ dominates $\tilde{f}(\tilde{x}^{*})$. Hence, $\tilde{x}^{*}$ is a non-dominated solution of ~$\min_{\tilde{x} \in X}~\tilde{f}$. \\ \\
The Part (\ref{part2}) can be proved similarly. \qed
\end{proof}

\begin{example}
Consider the fuzzy function given in Example \ref{ex42}. We see that $\tilde{f}^{\prime}(\tilde{x}) = 2 \odot \tilde{x}~ \ominus~ 4$ and  $\tilde{f}^{\prime\prime}(\tilde{x}) = 2$ which is a crisp function. Therefore, $\tilde{f}$ is two times continuously differentiable. With the help of the sufficient optimality condition, since $f_{1}^{\prime\prime}(x_{1}(x,\alpha), x_{2}(x,\alpha), \alpha) = 2 > 0$, $\tilde{x}^{*} = \tilde{2}$ is a non-dominated point of $\tilde{f}$.
\end{example}

%======================================================================================================================
\section{Conclusion}\label{section5}
In this paper, a new differentiability concept for fuzzy functions of fuzzy variable has been introduced. The algebraic properties of the derivatives have been explored. Besides, a necessary and sufficient optimality condition for a fuzzy optimization problem with fuzzy variables has been given, which is not yet studied in the literature. It is to note that the proposed optimality conditions are independent of the parameter $\alpha$ of the $\alpha$-cut and it makes the results easy to apply in different types of fuzzy optimization problems. In future we will attempt to identify a KKT-type optimality conditions for a constrained fuzzy optimization problem with fuzzy variable.

\section*{Funding} This research did not receive any specific grant from funding agencies in the public, commercial, or not-for-profit sectors.

      % American Physical Society (APS) style, author-year citations

\end{document}